\documentclass[12pt, reqno]{amsart}

\usepackage{fullpage}
\usepackage{amssymb}    
\usepackage{latexsym}   
\usepackage{url}        
\usepackage{array}      
\usepackage{enumitem}   

\theoremstyle{plain}
\newtheorem{theorem}{Theorem}[section]
\newtheorem{lemma}[theorem]{Lemma}
\newtheorem{proposition}[theorem]{Proposition}
\newtheorem{corollary}[theorem]{Corollary}

\newtheorem{problem}[theorem]{Problem}

\theoremstyle{remark}
\newtheorem{example}[theorem]{Example}

\newtheorem*{acknowledgment}{Acknowledgment}

\numberwithin{equation}{section}

\newcommand{\seclabel}[1]{\label{sec:#1}}   
\newcommand{\thmlabel}[1]{\label{thm:#1}}   
\newcommand{\lemlabel}[1]{\label{lem:#1}}   
\newcommand{\corlabel}[1]{\label{cor:#1}}   
\newcommand{\prplabel}[1]{\label{prp:#1}}   
\newcommand{\eqnlabel}[1]{\label{eqn:#1}}   

\newcommand{\thmref}[1]{\ref{thm:#1}}   
\newcommand{\lemref}[1]{\ref{lem:#1}}   
\newcommand{\corref}[1]{\ref{cor:#1}}   
\newcommand{\prpref}[1]{\ref{prp:#1}}   
\newcommand{\eqnref}[1]{\eqref{eqn:#1}} 

\newcommand{\setof}[2]{\{#1\,|\,#2\}}   

\newcommand{\inv}{^{-1}}                        
\newcommand{\sbl}[1]{\langle#1\rangle}      

\title{Half-isomorphisms of Moufang Loops}

\author[M. Kinyon]{Michael Kinyon}
\author[I. Stuhl]{Izabella Stuhl}
\author[P. Vojt\v{e}chovsk\'{y}]{Petr Vojt\v{e}chovsk\'{y}}

\address{Department of Mathematics \\
University of Denver \\
Denver, CO 80208 USA}
\email[Kinyon]{\url{mkinyon@du.edu}}
\email[Stuhl]{\url{izabella.stuhl@du.edu}}
\email[Vojt\v{e}chovsk\'{y}]{\url{petr@math.du.edu}}

\date{\today}

\subjclass[2000]{20N05}
\keywords{Moufang loop, half-isomorphism}

\begin{document}

\begin{abstract}
We prove that if the squaring map in the factor loop of a Moufang loop $Q$ over its nucleus is surjective, then every half-isomorphism of $Q$ onto a Moufang loop is either an isomorphism or an anti-isomorphism. This generalizes all earlier results in this vein.
\end{abstract}

\maketitle

\section{Introduction}
\seclabel{intro}

A \emph{loop} $(Q,\cdot)$ is a set $Q$ with a binary operation $\cdot$ such that for each $a$, $b\in Q$, the equations $a\cdot x=b$ and $y\cdot a=b$ have unique solutions $x$, $y\in Q$, and there exists a neutral element $1\in Q$ such that $1\cdot x = x\cdot 1 = x$ for all $x\in Q$. We will often write $xy$ instead of $x\cdot y$ and use $\cdot$ to indicate priority of multiplications. For instance, $xy\cdot z$ stands for $(x\cdot y)\cdot z$.

A \emph{Moufang loop} is a loop satisfying any (and hence all) of the Moufang identities
\[
xy\cdot zx = x(yz\cdot x),\qquad (xy\cdot x)z = x(y\cdot xz),\qquad (zx\cdot y)x = z(x\cdot yx)\,.
\]
In this paper we will only need the first Moufang identity, namely
\begin{equation}\eqnlabel{Moufang}
    xy\cdot zx = x(yz\cdot x)\,.
\end{equation}
Basic references to loop theory in general and Moufang loops in particular are \cite{Bruck, Pflugfelder}.

A loop is \emph{diassociative} if every subloop generated by two elements is associative (hence a group). By Moufang's theorem \cite{Moufang}, if three elements of a Moufang loop associate in some order, then they generate a subgroup. In particular, every Moufang loop is diassociative. We will drop further unnecessary parentheses while working with diassociative loops, for instance in the expression $xyx$.

If $Q$, $Q'$ are loops, a mapping $\varphi : Q\to Q'$ is a \emph{half-homomorphism} if, for every $x$, $y\in Q$, either $\varphi (xy) = \varphi x \cdot \varphi y$ or $\varphi (xy) = \varphi y\cdot \varphi x$. A bijective
half-homomorphism is a \emph{half-isomorphism}, and a \emph{half-automorphism} is defined as expected.

\medskip

The starting point for the investigation of half-isomorphisms of loops is the following result of Scott.

\begin{proposition}[\cite{Scott}, Theorem 1]\prplabel{scott}
Let $G$, $G'$ be groups. Every half-isomorphism of $G$ onto $G'$ is either an isomorphism or an anti-isomorphism.
\end{proposition}

Scott actually stated and proved Proposition \prpref{scott} in the more general situation where both $G$ and $G'$ are cancellative semigroups. Kuznecov \cite{Kuz} showed that the same conclusion holds if $G'$ is an arbitrary semigroup and $G$ is a semigroup containing a densely imbedded completely simple ideal. Proposition \prpref{scott} eventually became an exercise in Bourbaki's \textit{Algebra I} \cite[{\S}4, Exercise 26, p. 139]{Bo} with the steps in the exercise essentially following Scott's proof.

Scott gave an example of a loop of order 8 which shows that Proposition \prpref{scott} does not directly generalize to loops. It is nevertheless natural to ask if the result generalizes to Moufang loops, since these are highly structured loops that are, in some sense, very close to groups. This question was first addressed by Gagola and Giuliani \cite{GG1} who proved the following.

\begin{proposition}[\cite{GG1}]\prplabel{gg}
Let $Q$, $Q'$ be Moufang loops of odd order. Every half-isomorphism of $Q$ onto $Q'$ is either an isomorphism or an anti-isomorphism.
\end{proposition}

We will call a half-isomorphism which is neither an isomorphism nor an anti-isomorphism a \emph{proper} half-isomorphism. (In \cite{GG1,GG2}, the word ``nontrivial'' is used instead.)

Gagola and Giuliani also showed that there exist Moufang loops of even order with proper half-automorphisms \cite{GG2}.

The next result in the same vein was by Grishkov \emph{et al}. A loop is \emph{automorphic} if all of its inner mappings are automorphisms \cite{BP, KKPV}.

\begin{proposition}[\cite{GGRS}]\prplabel{ggrs}
Every half-automorphism of a finite automorphic Moufang loop is either an automorphism or an anti-automorphism.
\end{proposition}

Grishkov \emph{et al} conjectured that the finiteness assumption can be dropped, and that the corresponding result holds for all half-isomorphisms.

Our main result simultaneously generalizes Propositions \prpref{scott}, \prpref{gg} and \prpref{ggrs}, and as a by-product answers the conjecture of \cite{GGRS} in the affirmative.

To state the main result, we first recall that the \emph{nucleus} of a loop $Q$ is defined by
\[
N(Q) = \setof{a\in Q}{ax\cdot y = a\cdot xy,\, xa\cdot y = x\cdot ay,\, xy\cdot a = x\cdot ya,\ \forall x,\,y\in Q}\,.
\]
The nucleus is a subloop of any loop, and if $Q$ is a Moufang loop, then $N(Q)$ is a normal subloop \cite[Theorem 2.1]{Bruck}.

\begin{theorem}\thmlabel{main}
Let $Q$, $Q'$ be Moufang loops, and let $N=N(Q)$ be the nucleus of $Q$. Suppose that in $Q/N$ the squaring map $xN\mapsto x^2 N$ is surjective. Then every half-isomorphism of $Q$ onto $Q'$ is either an isomorphism or an anti-isomorphism.
\end{theorem}

To see why Theorem \thmref{main} generalizes Proposition \prpref{scott}, we note that if $Q$ is a group, then $N(Q) = Q$, and hence the squaring map in $Q/N(Q) = 1$ is trivially surjective. Since the order of any element
of a finite Moufang loop divides the order of the loop, a finite Moufang loop $Q$ of odd order certainly has a surjective (in fact, bijective) squaring map, and hence so does $Q/N(Q)$. Thus Theorem \thmref{main} generalizes Proposition \prpref{gg}. Finally, if $Q$ is an automorphic Moufang loop, then $Q/N(Q)$ has exponent $3$ by \cite[Corollary 2]{KKP}, and so Theorem \thmref{main} generalizes Proposition \prpref{ggrs} as well.

\medskip

We conclude this introduction with some motivational remarks. Scott's original result might seem at first to be a curiosity, but there is interest in it centered around the
result of Formanek and Sibley \cite{FS} that the group determinant determines a group. A shorter and more constructive proof of \cite{FS} was later given by Mansfield \cite{Man}. Hoehnke and Johnson generalized this to show that the $1$-, $2$-, and $3$-characters of a group determine the group \cite{HJ}. A more explicit use of the fact that group half-isomorphisms are either isomorphisms or anti-isomorphisms can be found in \cite{HJ}, which cites the aforementioned exercise in (the 1970 French edition of) \cite{Bo}.

Loops have determinants as well \cite{J}, and all the results on half-isomorphisms of Moufang loops are motivated
by the following open question: \emph{Let} $\mathcal{M}$ \emph{be a class of Moufang loops in which every half-isomorphism
is an isomorphism or an anti-isomorphism. Does the loop determinant of a loop in} $\mathcal{M}$ \emph{determine the loop?}

\section{Proper Half-isomorphisms}

Our goal in this section is Theorem \thmref{key}, which describes necessary conditions for the existence of a
proper half-isomorphism between Moufang loops. We start by expanding upon a lemma of Scott.

\begin{lemma}\lemlabel{basic}
Let $\varphi:Q\to Q'$ be a half-isomorphism of diassociative loops and let $x$, $y\in Q$.
\begin{enumerate}
\item[(i)] If $xy=yx$, then $\varphi x\cdot \varphi y = \varphi y\cdot\varphi x$.
\item[(ii)] If $\varphi(xy)=\varphi x\cdot \varphi y$, then $\varphi(yx) = \varphi y\cdot \varphi x$.
\item[(iii)] $\varphi(1)=1$ and $(\varphi x)^{-1} = \varphi(x^{-1})$.
\item[(iv)] If $\emptyset\ne X\subseteq Q$, then $\varphi(\langle X\rangle) = \langle \varphi x\,|\,x\in X\rangle$.
\end{enumerate}
\end{lemma}
\begin{proof}
For (i) and (ii), the proofs of steps (1) and (2) in \cite[Theorem 1]{Scott} go through word for word in the setting of diassociative loops. For (iii), note that $\varphi 1=\varphi(1\cdot 1) = \varphi 1\cdot \varphi 1$, so $\varphi 1 = 1$. Then $1 = \varphi 1 = \varphi( xx^{-1}) = \varphi x\cdot \varphi(x^{-1})$ by (i), so $\varphi(x^{-1}) = (\varphi x)^{-1}$.

Since left and right divisions can be expressed in terms of multiplication and inverses in diassociative loops, every element of $\langle X\rangle$ is a word $w$ involving only multiplications and inverses of elements from $X$, parenthesized in some way. Since $(\varphi x)^{-1} = \varphi(x^{-1})$ by (iii), we can assume that $X=X^{-1}$ and that no inverses occur in $w$. Suppose that $w$ has leaves $x_1$, $\dots$, $x_n\in X$, possibly with repetitions. Applying $\varphi$ to $w$ yields a term with leaves $\varphi(x_1)$, $\dots$, $\varphi(x_n)$ is some order. Therefore $\varphi(\langle X\rangle) \subseteq \langle \varphi x\,|\,x\in X\rangle$.

For the converse, consider a word $w$ in $\varphi(x_1)$, $\dots$, $\varphi(x_n)$. We prove by induction on the height of $w$ that $w\in \varphi(\langle X\rangle)$. If $w=\varphi x$, there is nothing to prove. Suppose that $w=\varphi u\cdot \varphi v$ for some $u$, $v\in\langle X\rangle$. If $w = \varphi(vu)$, we are done. Otherwise $\varphi(vu) = \varphi v\cdot \varphi u$, and (ii) implies $w = \varphi(uv)$.
\end{proof}

We can now generalize \cite[Lemma 3]{GG1}.

\begin{lemma}\lemlabel{gg-basic}
Let $\varphi : Q \to Q'$ be a half-isomorphism of diassociative loops. For all $a$, $b\in Q$:
\begin{enumerate}
\item[(i)] If $\varphi (ab) = \varphi a\cdot \varphi b$, then $\varphi\upharpoonright \sbl{a,b}$ is an isomorphism of groups.
\item[(ii)] If $\varphi (ab) = \varphi b\cdot \varphi a$, then $\varphi\upharpoonright \sbl{a,b}$ is an anti-isomorphism of groups.
\end{enumerate}
\end{lemma}
\begin{proof}
By Lemma \lemref{basic}, $\varphi(\sbl{a,b}) = \sbl{\varphi a,\varphi b}$, so $\psi = \varphi\upharpoonright \sbl{a,b}$ is a half-isomorphism of groups. By Proposition \prpref{scott}, $\psi$ is either an isomorphism or an anti-isomorphism. Suppose that $\varphi(ab) = \varphi a \cdot \varphi b$. If $\psi$ is an isomorphism, we are done. If $\psi$ is an anti-isomorphism then $\varphi b\cdot \varphi a  = \varphi(ab) = \varphi a\cdot \varphi b$, thus $\sbl{\varphi a,\varphi b}$ is commutative, and $\psi$ is also an isomorphism. Part (ii) follows similarly.
\end{proof}

A \emph{semi-homomorphism} $\varphi : Q\to Q'$ of diassociative loops is a mapping satisfying $\varphi(1) = 1$
and
\[
\varphi(xyx) = \varphi x\cdot \varphi y\cdot \varphi x
\]
for all $x$, $y\in Q$. From Lemma \lemref{gg-basic}, we immediately obtain a generalization of \cite[Corollary 4]{GG1}:

\begin{corollary}\corlabel{semi}
Every half-isomorphism of diassociative loops is a semi-isomorphism.
\end{corollary}

From now on, we will use diassociativity, the Moufang identity \eqnref{Moufang} and Corollary \corref{semi} without explicit reference.

\begin{lemma}
\lemlabel{subloops}
Let $\varphi : Q\to Q'$ be a half-isomorphism of Moufang loops $Q$, $Q'$, and let
\begin{align*}
    A &= \setof{a\in Q}{\varphi (ax) = \varphi a\cdot \varphi x\ \forall x\in Q}\,, \\
    B &= \setof{a\in Q}{\varphi (ax) = \varphi x\cdot \varphi a\ \forall x\in Q}\,.
\end{align*}
Then $A$ and $B$ are subloops of $Q$.
\end{lemma}
\begin{proof}
By Lemma \lemref{gg-basic}, it is clear that both $A$ and $B$ are closed under taking inverses $a\mapsto a\inv$. Thus it remains to show that both $A$ and $B$ are closed under multiplication. Fix $a$, $b\in Q$. Then for all $x\in Q$,
\begin{equation}\eqnlabel{auxsubloops}
    \varphi (ab\cdot x) = \varphi (ab\cdot (xa\inv\cdot a)) = \varphi (a (b\cdot xa\inv) a) = \varphi a\cdot \varphi (b\cdot xa\inv) \cdot \varphi a\,.
\end{equation}
If $a$, $b\in A$, then \eqnref{auxsubloops} yields
\begin{displaymath}
    \varphi (ab\cdot x) = \varphi a\cdot (\varphi b\cdot \varphi (xa\inv))\cdot \varphi a
    = (\varphi a\cdot \varphi b) (\varphi (xa\inv)\cdot \varphi a)
    = \varphi (ab)\cdot \varphi(xa\inv\cdot a)
    = \varphi (ab)\cdot \varphi x\,.
\end{displaymath}
On the other hand, if $a$, $b\in B$, then \eqnref{auxsubloops} gives
\begin{displaymath}
    \varphi (ab\cdot x) = \varphi a\cdot (\varphi (xa\inv)\cdot \varphi b)\cdot \varphi a
    = (\varphi a\cdot \varphi(xa\inv))\cdot (\varphi b\cdot \varphi a)
    = \varphi (xa\inv\cdot a)\cdot \varphi (ab)
    = \varphi x\cdot \varphi (ab)\,.
\end{displaymath}
\end{proof}

\begin{lemma}
\lemlabel{union}
No loop is the union of two proper subloops.
\end{lemma}
\begin{proof}
This is a standard exercise in group theory and the same proof holds here. For a contradiction, suppose that $A$, $B$ are proper subloops of a loop $Q$ with $Q = A\cup B$. Fix $a\in A\backslash B$ and $b\in B\backslash A$. We have $ab\in A$ or $ab\in B$ since $Q= A\cup B$. However, $ab\in A$ implies $b\in A$ since $A$ is closed under left division, and similarly $ab\in B$ implies $a\in B$.
\end{proof}

A version of the following result (without the third part of the conclusion) was proved in \cite[Proposition 5]{GG1}. Their proof used the general finiteness assumption of Proposition \prpref{gg} in an essential way. Our statement and proof make no reference to cardinality.

\begin{theorem}
\thmlabel{key}
Suppose that there is a proper half-isomorphism of Moufang loops $Q$, $Q'$. Then there is a proper half-isomorphism $\varphi : Q\to Q'$ and elements $a$, $b$, $c\in Q$ such that the following properties hold:
\begin{enumerate}
\item[(i)] $\varphi\upharpoonright \sbl{a,b}$ is an isomorphism and $\sbl{a,b}$ is nonabelian,
\item[(ii)] $\varphi\upharpoonright \sbl{a,c}$ is an anti-isomorphism and $\sbl{a,c}$ is nonabelian,
\item[(iii)] $\varphi\upharpoonright \sbl{b,c}$ is an isomorphism and $\sbl{b,c}$ is nonabelian.
\end{enumerate}
\end{theorem}
\begin{proof}
Assume $\varphi : Q\to Q'$ is a proper half-isomorphism, and let $A$, $B$ be defined as in Lemma \lemref{subloops}. Since $\varphi$ is proper, both $A$ and $B$ are proper subloops by Lemma \lemref{subloops}. By Lemma \lemref{union}, $Q \neq A\cup B$, so there is an element $a\in Q$ which is in neither subloop.

Since $a\not\in B$, there exists $b\in Q$ such that $\varphi (ab) = \varphi a\cdot \varphi b \neq \varphi b\cdot \varphi a$. By Lemma \lemref{gg-basic}, $\varphi\upharpoonright \sbl{a,b}$ is an isomorphism and $\sbl{a,b}$ is nonabelian, proving (i).

Since $a\not\in A$, there exists
$c\in Q$ such that $\varphi (ac) = \varphi c\cdot \varphi a \neq \varphi a\cdot \varphi c$. By Lemma \lemref{gg-basic}, $\varphi\upharpoonright \sbl{a,c}$ is an anti-isomorphism and $\sbl{a,c}$ is nonabelian, proving (ii).

Let $J : x\mapsto x\inv$ denote the inversion permutation on $Q$. Now, $\varphi\upharpoonright \sbl{b,c}$ is either an isomorphism or an anti-isomorphism by Lemma \lemref{gg-basic}. If the latter case holds,
then $\varphi\circ J$ is still a proper half-isomorphism, $(\varphi\circ J) \upharpoonright \sbl{b,c}$ is an isomorphism, $(\varphi\circ J)\upharpoonright \sbl{a,b}$ is an anti-isomorphism and $(\varphi\circ J)\upharpoonright \sbl{a,c}$ is an isomorphism. In particular, conditions (i) and (ii) hold for $\varphi\circ J$ with the roles of $b$ and $c$ reversed. Thus there is no loss of generality in assuming that
$\varphi\upharpoonright \sbl{b,c}$ is an isomorphism.

Next we compute
\begin{displaymath}
    \varphi(cb\cdot ac) = \varphi(c\cdot ba\cdot c)
    = \varphi c\cdot \varphi (ba)\cdot \varphi c
    = \varphi c\cdot (\varphi b\cdot \varphi a)\cdot \varphi c
    = (\varphi c\cdot \varphi b) (\varphi a\cdot \varphi c)
    = \varphi (cb)\cdot \varphi(ca)\,.
\end{displaymath}
If also $\varphi (cb)\cdot \varphi(ca) = \varphi(cb\cdot ca)$, then $cb\cdot ac = cb\cdot ca$, and so $ac = ca$, a contradiction. Therefore $\varphi(cb\cdot ac) = \varphi(cb)\cdot \varphi(ca) = \varphi(ca\cdot cb)$, and this establishes $cb\cdot ac = ca\cdot cb$. If $\sbl{b,c}$ were an abelian group, then we would have $c\cdot ba\cdot c = cb\cdot ac = ca\cdot cb = ca\cdot bc = c\cdot ab\cdot c$, which implies $ba = ab$, a contradiction. This establishes the remaining claim in (iii) and completes the proof of the theorem.
\end{proof}

Given a proper half-isomorphism $\varphi : Q\to Q'$ of Moufang loops, we will refer to a triple $(a,b,c)$ of elements $a$, $b$, $c\in Q$ satisfying the conditions of Theorem \thmref{key} as a \emph{Scott triple}, since the idea of considering triples satisfying conditions (i) and (ii) of the theorem goes back to Scott's original paper \cite{Scott}.

\begin{theorem}
\thmlabel{woohoo}
Let $\varphi : Q\to Q'$ be a proper half-isomorphism of Moufang loops, and let $(a,b,c)$ be a Scott triple.
Then $\sbl{a^2,c}$ and $\sbl{a,c^2}$ are abelian groups.
\end{theorem}

\begin{proof}
Note that if $(a,b,c)$ is a Scott triple then $(c,b,a)$ is also a Scott triple. It therefore suffices to verify that $\sbl{a^2,c}$ is an abelian group.
For $k\geq 1$, we calculate
\begin{displaymath}
    \varphi(b\cdot a^k c\cdot b)
    = \varphi b\cdot \varphi(a^k c)\cdot \varphi b
    = \varphi b\cdot (\varphi c\cdot \varphi(a^k))\cdot \varphi b
    = (\varphi b\cdot \varphi c)(\varphi(a^k)\cdot \varphi b)
    = \varphi (bc)\cdot \varphi (a^k b)\,,
\end{displaymath}
and so
\begin{equation}\eqnlabel{keytmp0}
    \varphi(b\cdot a^k c\cdot b) = \varphi(bc)\cdot \varphi(a^k b)\,.
\end{equation}

Suppose for a while that $\varphi(bc)\cdot \varphi(ab) = \varphi(bc\cdot ab) = \varphi(b\cdot ca\cdot b)$. Comparing this with \eqnref{keytmp0} for $k=1$ yields $ac=ca$, a contradiction. Thus $\varphi(bc)\cdot \varphi(ab) = \varphi(ab\cdot bc)$, and so $b\cdot ac\cdot b = ab\cdot bc$. Rearranging this gives $ac = b\inv(ab\cdot bc)b\inv = b\inv ab\cdot bcb\inv$, and thus
\begin{equation}\eqnlabel{keytmp2}
    a^2 c = a(b\inv ab\cdot bcb\inv)\,.
\end{equation}

Now suppose for a while that $\varphi(bc)\cdot \varphi(a^2 b) = \varphi(a^2 b\cdot bc)$. Comparing this with \eqnref{keytmp0} for $k=2$ yields $b\cdot a^2 c\cdot b = a^2 b\cdot bc$, and so
\begin{equation}\eqnlabel{keytmp1}
    a^2 c = b\inv (a^2 b\cdot bc) b\inv = b\inv a^2 b\cdot bcb\inv = (b\inv a b)^2 \cdot bcb\inv = b\inv a b\cdot (b\inv ab\cdot bcb\inv)\,.
\end{equation}
Comparing \eqnref{keytmp2} and \eqnref{keytmp1}, we conclude that $b\inv a b = a$, that is, $ab = ba$, a contradiction. Therefore $\varphi(bc)\cdot \varphi(a^2 b) = \varphi(bc\cdot a^2 b)$ and
\[
\varphi(b\cdot a^2 c\cdot b) = \varphi(bc)\cdot \varphi(a^2 b) = \varphi(bc\cdot a^2 b) = \varphi(b\cdot ca^2\cdot b)\,,
\]
which implies $a^2 c = c a^2$.
\end{proof}

\section{Proof of the Main Theorem}

In this section we prove Theorem \thmref{main}. For a contradiction, suppose that $Q$, $Q'$ are Moufang loops, $N=N(Q)$ is the nucleus of $Q$, the squaring map in $Q/N$ is surjective, and let $\varphi : Q\to Q'$ be a proper half-isomorphism. By Theorem \thmref{key}, $Q$ contains a Scott triple $(a,b,c)$. By the assumption on $Q/N$, there is $d\in Q$ such that $d^2N = (d N)^2 = aN$, and so there is also $n\in N$ such that $d^2 = a n$.

Throughout the proof we will use the observation that $\langle n, x, y\rangle$ is a subgroup of $Q$ for any $x$, $y\in Q$, thanks to Moufang's theorem.

\begin{lemma}\lemlabel{auxmain}
In the above situation, $(d,b,c)$ is a Scott triple, and $\sbl{n,d}$ and $\sbl{n,c}$ are abelian groups.
\end{lemma}
\begin{proof}
Since $\sbl{n,d,b}$ is a group, $\sbl{a,b}$ is nonabelian, $\sbl{a,b}\leq \sbl{n,d,b}$ and $\varphi\upharpoonright \sbl{a,b}$ is an isomorphism, we have that $\varphi\upharpoonright \sbl{n,d,b}$ is an isomorphism, and therefore $\varphi\upharpoonright \sbl{d,b}$ is an isomorphism.

Since $\sbl{n,d,c}$ is a group, $\sbl{a,c}$ is nonabelian, $\sbl{a,c}\leq \sbl{n,d,c}$ and $\varphi\upharpoonright \sbl{a,c}$ is an anti-isomorphism, we have that $\varphi\upharpoonright \sbl{n,d,c}$ is an anti-isomorphism, and therefore $\varphi\upharpoonright \sbl{d,c}$ is an anti-isomorphism.

Since $\sbl{n,b,c}$ is a group, $\sbl{b,c}$ is nonabelian, $\sbl{b,c}\leq \sbl{n,b,c}$ and $\varphi\upharpoonright \sbl{b,c}$ is an isomorphism, we have that $\varphi\upharpoonright \sbl{n,b,c}$ is an isomorphism.

Now $\varphi\upharpoonright \sbl{n,d}$ is both an isomorphism and an anti-isomorphism, and so $\sbl{n,d}$ is an abelian group. Also, $\varphi\upharpoonright \sbl{n,c}$ is both an isomorphism and
an anti-isomorphism, and so $\sbl{n,c}$ is an abelian group.

If $\sbl{c,d}$ were abelian, then from the above it would follow that $\sbl{n,c,d}$ is abelian, contradicting the fact that $\sbl{a,c}$ is nonabelian. Thus
$\sbl{c,d}$ is nonabelian.

Suppose $\sbl{d,b}$ is abelian. We calculate
\begin{displaymath}
    \varphi( db\cdot cd ) = \varphi(d\cdot bc\cdot d)
    = \varphi d\cdot \varphi(bc)\cdot \varphi d
    = \varphi d\cdot (\varphi b\cdot \varphi c)\cdot \varphi d
    = (\varphi d\cdot \varphi b)(\varphi c\cdot \varphi d)
    = \varphi (db)\cdot \varphi(dc)\,.
\end{displaymath}
If $\varphi(db)\cdot \varphi(dc) = \varphi(db\cdot dc)$, then $\varphi(db\cdot cd) = \varphi(db\cdot dc)$, and so
$db\cdot cd = db\cdot dc$. This gives $cd = dc$, a contradiction. On the other hand, if $\varphi(db\cdot cd) = \varphi(dc\cdot db)$, then
$d\cdot bc\cdot d = dc\cdot db = dc\cdot bd = d\cdot cb\cdot b$, so that $bc = cb$, another contradiction. Therefore $\sbl{d,b}$ is nonabelian.
\end{proof}

Let us now finish the proof of Theorem \thmref{main}. By Lemma \lemref{auxmain}, $\sbl{n,c}$ is abelian, $\sbl{n,d}$ is abelian, and $(d,b,c)$ is a Scott triple. By Theorem \thmref{woohoo}, $\sbl{d^2,c} = \sbl{an,c}$ is abelian. Finally, $na = nd^2n^{-1} = d^2 = an$ because $\sbl{n,d}$ is abelian, so $\sbl{n,a} = \sbl{n,an}$ is abelian. Altogether, $\sbl{n,an,c}$ is an abelian group. But then $\sbl{a,c}\le \sbl{n,an,c}$ is abelian, a contradiction with $(a,b,c)$ being a Scott triple.

\section{Remarks and open problems}

In this section we examine hypotheses and generalizations of Theorem \thmref{main}.

The somewhat technical assumption of Theorem \thmref{main} that squaring in $Q/N(Q)$ is surjective can be replaced by the assumption that squaring in $Q/N(Q)$ is bijective. We do not know if these two assumptions are equivalent in Moufang loops.

\begin{problem}
Is there a Moufang loop $Q$ with nucleus $N(Q)$ such that the squaring map in $Q/N(Q)$ is surjective but not injective?
\end{problem}

Scott extended Proposition \prpref{scott} to half-homomorphisms of groups by showing that the kernel of a half-homomorphism is a normal subloop and that the given half-homomorphism factors through the quotient via the
natural homomorphism \cite[Theorem 2]{Scott}.

It is easy to show that the kernel of a half-homomorphism of loops is a subloop:

\begin{lemma}
Let $\varphi:Q\to Q'$ be a half-homomorphism of loops. Then $\mathrm{Ker}(\varphi) = \{a\in Q\,|\,\varphi a =1\}$ is a subloop of $Q$.
\end{lemma}
\begin{proof}
Let $K=\mathrm{Ker}(\varphi)$ and $a$, $b\in K$. Then $\varphi(ab)\in\{\varphi a\cdot \varphi b$, $\varphi b\cdot \varphi a\} = \{1\}$, so $ab\in K$. Denote by $a/b$ the unique element of $Q$ such that $(a/b)b = a$. Then $1 = \varphi a = \varphi( (a/b)b)$ is equal to $\varphi(a/b)\cdot \varphi b = \varphi(a/b)$ or to $\varphi b\cdot \varphi(a/b) = \varphi(a/b)$. In either case, $a/b\in K$ follows. Similarly for the left division.
\end{proof}

However, we do not know the answer to the following problem and hence whether Scott's result on kernels can be generalized:

\begin{problem}
Let $\varphi : Q\to Q'$ be a half-homomorphism of a Moufang loop $Q$ into a loop $Q'$. Is the kernel of $\varphi$ necessarily normal in $Q$?\end{problem}

The present paper and all proofs in this context rely rather heavily on the assumption that one is working with Moufang loops. The following problem therefore suggests itself.

\begin{problem}
Investigate half-isomorphisms in other classes of loops, such as Bol loops or automorphic loops.
\end{problem}

For automorphic loops, Grishkov \emph{et al} gave an example with a proper half-automorphism \cite[Example 2]{GGRS}, so any reasonable theorem will have to restrict the hypotheses further, similarly to the case of Moufang loops. The authors are not aware of any work on half-isomorphisms of Bol loops that are not Moufang.

\bigskip

Finally, the following example shows that the target $Q'$ of a half-isomorphism $Q\to Q'$ need not be a diassociative loop even if $Q$ is a group.

\begin{example}
The identity mapping on the set $Q=\{0,\dots,5\}$ is a half-isomorphism $(Q,\cdot)\to (Q,*)$ of the respective loops
\begin{displaymath}
\begin{array}{c|cccccc}
    \cdot & 0 & 1 & 2 & 3 & 4 & 5 \\
    \hline
    0 & 0 & 1 & 2 & 3 & 4 & 5\\
    1 & 1 & 2 & 0 & 5 & 3 & 4\\
    2 & 2 & 0 & 1 & 4 & 5 & 3\\
    3 & 3 & 4 & 5 & 0 & 1 & 2\\
    4 & 4 & 5 & 3 & 2 & 0 & 1\\
    5 & 5 & 3 & 4 & 1 & 2 & 0
\end{array}
\qquad
\begin{array}{c|cccccc}
    * & 0 & 1 & 2 & 3 & 4 & 5 \\
    \hline
    0 & 0 & 1 & 2 & 3 & 4 & 5\\
    1 & 1 & 2 & 0 & 4 & 5 & 3\\
    2 & 2 & 0 & 1 & 5 & 3 & 4\\
    3 & 3 & 5 & 4 & 0 & 1 & 2\\
    4 & 4 & 3 & 5 & 2 & 0 & 1\\
    5 & 5 & 4 & 3 & 1 & 2 & 0
\end{array}
\quad .
\end{displaymath}
Here, $(Q,\cdot)$ is the symmetric group $S_3$, and $(Q,*)$ is an automorphic loop that is not diassociative, as witnessed by $3*(3*1) = 3*5 = 2\ne 1 = 0*1 = (3*3)*1$.
\end{example}


\begin{acknowledgment}
Our investigations were aided by the automated deduction tool \textsc{Prover9} developed by McCune \cite{McCune}.
\end{acknowledgment}


\end{document}